\documentclass[oneside,english]{amsart}
\usepackage[utf8]{inputenc}
\usepackage{amsthm}
\usepackage{amstext}
\usepackage{amssymb}

\makeatletter
\numberwithin{equation}{section}
\numberwithin{figure}{section}
\theoremstyle{plain}
\newtheorem{thm}{\protect\theoremname}[section]
  \theoremstyle{remark}
  \newtheorem{rem}[thm]{\protect\remarkname}
  \theoremstyle{definition}
  \newtheorem{defn}[thm]{\protect\definitionname}
  \theoremstyle{plain}
  \newtheorem{lem}[thm]{\protect\lemmaname}
  \theoremstyle{remark}
  \newtheorem{claim}[thm]{\protect\claimname}
  \theoremstyle{plain}
  \newtheorem{prop}[thm]{\protect\propositionname}
  \theoremstyle{plain}
  \newtheorem{cor}[thm]{\protect\corollaryname}
  \theoremstyle{plain}
  \newtheorem{question}[thm]{\protect\questionname}



\makeatother

\usepackage{babel}
  \providecommand{\claimname}{Claim}
  \providecommand{\corollaryname}{Corollary}
  \providecommand{\definitionname}{Definition}
  \providecommand{\lemmaname}{Lemma}
  \providecommand{\propositionname}{Proposition}
  \providecommand{\questionname}{Question}
  \providecommand{\remarkname}{Remark}
\providecommand{\theoremname}{Theorem}

\begin{document}

\title[Mass formulas and quotient singularities II]{Mass formulas for local Galois representations and quotient singularities
II: dualities and resolution of singularities}

\author{Melanie Machett Wood}

\address{Department of Mathematics, University of Wisconsin-Madison, 480 Lincoln
Drive, Madison, WI 53705 USA, and American Institute of Mathematics,
360 Portage Ave, Palo Alto, CA 94306-2244 USA}

\email{mmwood@math.wisc.edu}

\author{Takehiko Yasuda}

\address{Department of Mathematics, Graduate School of Science, Osaka University,
Toyonaka, Osaka 560-0043, Japan, and Max Planck Institute for Mathematics, Vivatsgasse 7, 53111 Bonn, Germany}

\email{takehikoyasuda@math.sci.osaka-u.ac.jp, highernash@gmail.com}

\maketitle
\global\long\def\AA{\mathbb{A}}
\global\long\def\PP{\mathbb{P}}
\global\long\def\NN{\mathbb{N}}
\global\long\def\GG{\mathbb{G}}
\global\long\def\ZZ{\mathbb{Z}}
\global\long\def\QQ{\mathbb{Q}}
\global\long\def\CC{\mathbb{C}}
\global\long\def\FF{\mathbb{F}}
\global\long\def\LL{\mathbb{L}}
\global\long\def\RR{\mathbb{R}}
\global\long\def\DD{\mathbb{D}}

\global\long\def\bx{\mathbf{x}}
\global\long\def\bf{\mathbf{f}}
\global\long\def\ba{\mathbf{a}}
\global\long\def\bs{\mathbf{s}}
\global\long\def\bt{\mathbf{t}}
\global\long\def\bw{\mathbf{w}}
\global\long\def\bb{\mathbf{b}}
\global\long\def\bv{\mathbf{v}}
\global\long\def\bp{\mathbf{p}}

\global\long\def\cN{\mathcal{N}}
\global\long\def\cW{\mathcal{W}}
\global\long\def\cY{\mathcal{Y}}
\global\long\def\cM{\mathcal{M}}
\global\long\def\cF{\mathcal{F}}
\global\long\def\cX{\mathcal{X}}
\global\long\def\cE{\mathcal{E}}
\global\long\def\cJ{\mathcal{J}}
\global\long\def\cO{\mathcal{O}}
\global\long\def\cD{\mathcal{D}}
\global\long\def\cZ{\mathcal{Z}}
\global\long\def\cK{\mathcal{K}}
\global\long\def\cP{\mathcal{P}}
\global\long\def\cR{\mathcal{R}}

\global\long\def\fs{\mathfrak{s}}
\global\long\def\fp{\mathfrak{p}}
\global\long\def\fm{\mathfrak{m}}

\global\long\def\Spec{\mathrm{Spec}\,}
\global\long\def\Hom{\mathrm{Hom}}
\global\long\def\Tr{\mathrm{Tr}}
\global\long\def\mot{\mathrm{mot}}

\global\long\def\Var{\mathrm{Var}}
\global\long\def\tame{\mathrm{tame}}
\global\long\def\Gal{\mathrm{Gal}}
\global\long\def\Jac{\mathrm{Jac}}
\global\long\def\Ker{\mathrm{Ker}}
\global\long\def\Im{\mathrm{Im}}
\global\long\def\Aut{\mathrm{Aut}}
\global\long\def\st{\mathrm{st}}
\global\long\def\diag{\mathrm{diag}}
\global\long\def\characteristic{\mathrm{char}}
\global\long\def\tors{\mathrm{tors}}
\global\long\def\sing{\mathrm{sing}}
\global\long\def\red{\mathrm{red}}
\global\long\def\Ind{\mathrm{Ind}}
\global\long\def\nr{\mathrm{nr}}
\global\long\def\rep{\mathrm{rep}}
\global\long\def\good{\mathrm{good}}
 \global\long\def\univ{\mathrm{univ}}
\global\long\def\length{\mathrm{length}}
\global\long\def\sm{\mathrm{sm}}
\global\long\def\top{\mathrm{top}}
\global\long\def\rank{\mathrm{rank}}
\global\long\def\Mot{\mathrm{Mot}}
\global\long\def\age{\mathrm{age}\,}
\global\long\def\et{\mathrm{et}}
\global\long\def\hom{\mathrm{hom}}
\global\long\def\tor{\mathrm{tor}}
\global\long\def\KWRep{K_{0}(\mathrm{WRep}_{k,l})}
\global\long\def\chiet{\chi_{\mathrm{\acute{e}t}}}
\global\long\def\hatKWRep{\widehat{K_{0}(\mathrm{WRep}_{k,l})}}
\global\long\def\ord{\mathrm{ord}}
\global\long\def\chist{\chi_{\st}}

\global\long\def\Conj#1{\mathrm{Conj}(#1)}
\global\long\def\Mass#1{\mathrm{Mass}(#1)}
\global\long\def\Inn#1{\mathrm{Inn}(#1)}
\global\long\def\bConj#1{\mathbf{Conj}(#1)}
\global\long\def\Hilb{\mathrm{Hilb}}
\global\long\def\sep{\mathrm{sep}}
\global\long\def\GL#1#2{\mathrm{GL}_{#1}(#2)}
\global\long\def\codim{\mathrm{codim}}
\global\long\def\TrF{\mathrm{TrF}}
\global\long\def\Met{M_{\mathrm{\acute{e}t}}}
\global\long\def\Stab{\mathrm{Stab}}
\global\long\def\AF{\mathrm{AF}}

\begin{abstract}
A total mass is the weighted count of continuous homomorphisms from
the absolute Galois group of a local field to a finite group. In the
preceding paper, the authors observed that in a particular example,
two total masses coming from two different weightings are dual to
each other. We discuss the problem how general such a duality holds
and relate it to the existence of simultaneous resolution of singularities,
using the wild McKay correspondence and the Poincaré duality for stringy invariants. 
We also exhibit several examples. 
\end{abstract}

\section{Introduction}

In the preceding paper \cite{Wood-Yasuda-I} the authors found a close
relationship between mass formulas for local Galois representations,
studied in the number theory, and stringy invariants of singular varieties.
In the present paper, we discuss dualities of mass formulas in relation
with the existence of some kinds of desingularization.

For a local field $K$ and a finite group $\Gamma$, let $G_{K}=\Gal(K^{\sep}/K)$
be the absolute Galois group of $K$ and $S_{K,\Gamma}$ be the set
of continuous homomorphisms $G_{K}\to\Gamma$. For a function $c:S_{K,\Gamma}\to\RR$,
the \emph{total mass }of $(K,\Gamma,c)$ is defined as
\[
M(K,\Gamma,c):=\frac{1}{\sharp\Gamma}\sum_{\rho\in S_{K,\Gamma}}q^{-c(\rho)}
\]
with $q$ the cardinality of the residue field of $K$. 
For the symmetric group $S_n$, 
fixing the
standard representation 
\[
\iota:S_{n}\hookrightarrow \GL n{\cO_{K}}\subset\GL nK
\]
with $\cO_{K}$ the integer ring of $K$, we can associate the Artin
conductor $\ba_{\iota}(\rho)$ to each continuous homomorphism $\rho:G_{K}\to S_{n}$.
According to Kedlaya \cite{MR2354797}, Bhargava's mass formula \cite{MR2354798}
for étale extensions of a local field is expressed as
\[
M(K,S_{n},\ba_{\iota})=\sum_{m=0}^{n-1}P(n,n-m)q^{-m},
\]
where $P(n,n-m)$ is the number of partitions of the integer $n$
into exactly $n-m$ parts. It was found in \cite{Wood-Yasuda-I} that
for another function $\rho\mapsto\bw_{2\iota}(\rho)$ originating
in the wild McKay correspondence \cite{Yasuda:2013fk}, we have 
\[
M(K,S_{n},-\bw_{2\iota})=\sum_{m=0}^{n-1}P(n,n-m)q^{m}.
\]
Here $2\iota$ stands for the direct sum of two copies of the representation
$\iota$. The right hand sides of the two formulas are interchanged
by replacing $q$ with $q^{-1}$: this is what we call a \emph{duality.}
It is then natural to ask how general the duality holds: what about
other groups and other representations. There exists yet another function
$\bv_{\tau}$ on $S_{K,\Gamma}$ for each representation $\tau$ of
$\Gamma$ over $\cO_{K}$. It was shown in \cite{Wood-Yasuda-I} that
if $\tau$ is a permutation representation, then we have $\ba_{\tau}=\bv_{2\tau}$.
It turns out that the function $\bv_{\tau}$ is more appropriate to
discuss dualities for non-permutation representations. The most basic
question we would like to ask is: \emph{when are the total masses
$M(K,\Gamma,\bv_{\tau})$ and $M(K,\Gamma,-\bw_{\tau})$ dual to each
other in the same way as $M(K,S_{n},\ba_{\iota})$ and $M(K,S_{n},-\bw_{2\iota})$
are? }We will observe that the duality does not always hold, but is
closely related to the existence of a \emph{simultaneous resolution
of singularities} or to \emph{equisingularities}. 

To discuss the duality more rigorously, we need to consider total
masses for all unramified extensions of the given local field $K$.
For each integer $r>0$, let $K_{r}$ be the unramified extension
of $K$ of degree $r$. For a representation $\tau:\Gamma\to\GL d{\cO_{K}}$,
we can naturally generalize functions $\bv_{\tau}$ and $\bw_{\tau}$
to ones on $S_{K_{r},\Gamma}$, which we continue to denote by the
same symbols. We then regard total masses $M(K_{r},\Gamma,\bv_{\tau})$
and $M(K_{r},\Gamma,-\bw_{\tau})$ as functions in the variable $r\in\NN$.
When they belong to a certain class of nice functions (which we call
\emph{admissible}), we can define their dual functions $\DD(M(K_{r},\Gamma,\bv_{\tau}))$
and $\DD(M(K_{r},\Gamma,-\bw_{\tau}))$ in such a way that the dual
of the function $q^{r}$ is $q^{-r}$. A duality which we are interested
in is now expressed by the equality
\begin{equation}
\DD(M(K_{r},\Gamma,\bv_{\tau}))=M(K_{r},\Gamma,-\bw_{\tau}).\label{eqintro1}
\end{equation}
We call it the \emph{strong duality}.

It is the \emph{wild McKay correspondence} which relates total masses
to singularities. Given a representation $\tau:\Gamma\to\GL d{\cO_{K}}$,
we have the associated $\Gamma$-action on the affine space $\AA_{\cO_{K}}^{d}$
and the quotient scheme $X:=\AA_{\cO_{K}}^{d}/\Gamma$, which is a
normal $\QQ$-Gorenstein variety over $\cO_{K}$. In general, for
a normal $\QQ$-Gorenstein variety $Y$ over $\cO_{K}$, we can define
the \emph{stringy point count }$\sharp_{\st}(Y)$ as a volume of the
$\cO_{K}$-point set $Y(\cO_{K})$ (see \cite{wild-p-adic}). This
is an analogue of stringy $E$-function defined by Batyrev and Dais  \cite{MR1404917, MR1672108}
and is a generalization of the number of $k$-points on a smooth $\cO_{K}$-variety.
For a $k$-point $y\in Y(k)$, we can similarly define the \emph{stringy
point count along $\{y\}$ (or the stringy weight of $y$)}, denoted
by $\sharp_{\st}(Y)_{y}$, so that we have 
\[
\sharp_{\st}(Y)=\sum_{y\in Y(k)}\sharp_{\st}(Y)_{y}.
\]
Yasuda \cite{wild-p-adic} proved that if the representation $\tau$
is faithful and the morphism $\AA_{\cO_{K}}^{d}\to X$ is étale in
codimension one, then 
\begin{gather*}
\sharp_{\st}(X)=M(K,\Gamma,\bv_{\tau})q^{d}\text{ and} \\
\sharp_{\st}(X)_{o}=M(K,\Gamma,-\bw_{\tau}),
\end{gather*}
where $o$ is the origin of $X(k)$. This is a version of the wild
McKay correspondence discussed in \cite{MR3230848,Yasuda:2013fk,Wood-Yasuda-I,Yasuda:2014fk2}.
Special cases were previously proved in \cite{MR3230848,Wood-Yasuda-I}.
If $X$ has a nice resolution, then $\sharp_{\st}(X)$ and $\sharp_{\st}(X)_{o}$
are explicitly computed in terms of resolution data. Using it, we
can deduce a few properties of the functions $M(K_{r},\Gamma,\bv_{\tau}$)
and $M(K_{r},\Gamma,-\bw_{\tau})$: obtained properties depend on
what sort of resolution exists.
\begin{rem}
The assumption that the morphism $\AA_{\cO_{K}}^{d}\to X$ is étale
in codimension one is just for notational simplicity. We may drop
this assumption by considering the pair of $X$ and a $\QQ$-divisor
on it rather than the variety $X$ itself.
\end{rem}
Let us think of $X$ as a family of singular varieties over $\Spec\cO_{K}$.
If $X$ admits a kind of simultaneous resolution, then 
\[
\frac{\sharp_\st (X)-\sharp_\st (X)_o}{q^r-1}
\]
satisfies a certain self-duality, which can be understood as the Poincaré duality 
of stringy invariants proved in \cite{MR1404917, MR1672108} over complex numbers. 
Using the wild McKay correspondence, we can transform this duality into the form:
\begin{multline}
M(K_{r},\Gamma,\bv_{\tau})\cdot q^{rd}-M(K_{r},\Gamma,-\bw_{\tau})\\
=\DD(M(K_{r},\Gamma,-\bw_{\tau}))\cdot q^{rd}-\DD(M(K_{r},\Gamma,\bv_{\tau})).\label{eqin2}
\end{multline}
We call it the \emph{weak duality}
between $M(K_{r},\Gamma,\bv_{\tau})$ and $M(K_{r},\Gamma,-\bw_{\tau})$, which
is indeed weaker than the strong duality (\ref{eqintro1})

There are three possibilities; both  dualities hold, both dualities fail, and the weak one holds but the strong one does not.  
Examples we compute show that all of them indeed occur. 
In the tame case, the strong duality holds. For permutation representations, 
the weak duality holds in all examples we compute, but the strong duality does not generally hold.
When $\cO_{K}=\FF_{q}[[t]]$ and the given representation $\tau$ is defined over $\FF_q$, 
the strong duality holds in all computed examples. 
However, when the representation is not defined over $\FF_q$, then 
the weak duality tends to fail. 
These examples suggest that 
we may think of the two dualities as a test of how equisingular the family $X$ is. 

In particular, when $\cO_{K}=\FF_{q}[[t]]$ and $\tau$ is defined over $\FF_q$, the associated 
quotient scheme $X =\AA ^d_{\FF_q[[t]]}$ would be equisingular in any reasonable sense.
If there exists a nice resolution of the $\FF_{q}$-variety $X\otimes_{\cO_{K}}\FF_{q}$,
then we obtain an equally nice simultaneous resolution of $X$ just
by the base change. Therefore, if one believes that there exists as
a nice resolution in positive characteristic as in characteristic
zero, then at least the weak duality (\ref{eqin2}) would hold for such $\tau$. On
the contrary, if one finds an example such that the weak duality fails,
then it would give an example of an $\FF_{q}$-variety without admitting
any nice resolution. 
However, as far as we computed, all examples satisfy even the strong duality. 

The study in this paper thus gives rise to interesting open problems related to
dualities and equisingularities and their interaction (see Questions \ref{Q:positive_char_weak_dual}, \ref{Q:positive_char_strong_dual} and \ref{Q:permutation_dual} and Section \ref{sec:Concluding_remarks}). Further studies are required.

The paper is organized as follows. In Section \ref{sec:Admissible-functions},
we define admissible functions and their duals. In Section \ref{sec:stringy},
we discuss properties of stringy point counts. In Section \ref{sec:Total-masses},
we define total masses and discuss how the wild McKay correspondence
relate dualities of total masses with equisingularities. In Section
\ref{sec:Examples}, we exhibit several examples.
In the final section, we give concluding remarks.

\subsection{Acknowledgments}

The authors would like to thank Kiran Kedlaya for helpful discussion. 
Yasuda thanks Max Planck Institute for Mathematics for its hospitality, where he stayed
partly during this work.
Wood was supported by NSF grants DMS-1147782 and DMS-1301690 and an American Institute of Mathematics Five Year Fellowship.

\subsection{Convention and notation}

We denote the set of positive integers by $\NN$. A \emph{local field}
means a finite extension of either $\QQ_{p}$ or $\FF_{p}((t))$ for
a prime number $p$. For a local field $K$, we denote its integer
ring by $\cO_{K}$ and its residue field by $k$. We denote the characteristic
of $k$ by $p$ and the cardinality of $k$ by $q$. For $r\in\NN$,
we denote the unramified extension of $K$ of degree $r$ by $K_{r}$
and its residue field (that is, the extension of $k$ of degree $r$)
by $k_{r}$. For a variety $X$ defined over either $\cO_{K}$ or
$k$, we put $\sharp X:=\sharp X(k)$, the cardinality of $k$-points
of $X$. More generally, for $r\in\NN$, we put $\sharp^{r}X:=\sharp X(k_{r})$.

\section{Admissible functions and their duals\label{sec:Admissible-functions}}

In this section, we set foundation to discuss dualities of total masses
and stringy point counts by introducing \emph{admissible functions.} 
\begin{defn}
We call a function $f:\NN\to\CC$ \emph{admissible }if there exist
numbers $0\ne c\in\QQ$, $n_{i}\in\ZZ$ and $\alpha_{i}\in\CC$ $(1\le i\le m)$
such that for every $r\in\NN$, 
\[
f(r)=\frac{\sum_{i=1}^{l}n_{i}\alpha_{i}^{r}}{q^{cr}-1}.
\]
We denote the set of admissible functions by $\AF$. For an admissible
function $f(r)$ as above, we define its \emph{dual function }$\DD f:\NN\to\CC$
by
\[
(\DD f)(r):=\frac{\sum_{i=1}^{l}n_{i}\alpha_{i}^{-r}}{q^{-cr}-1}.
\]

\end{defn}
It is easy to see that the set of admissible functions is closed under
addition and multiplication, thus $\AF$ has a natural ring structure.
For $f\in\AF$, the dual $\DD f$ is also an admissible function.
Therefore $\DD$ gives an involution of the set $\AF$. Moreover it
is easily checked that $\DD:\AF\to\AF$ is a ring isomorphism. 

A typical admissible function is a rational function in $q^{1/n}$
for some $n\in\NN$ with a denominator of the form $q^{cr}-1$. In
that case, the dual function is obtained just by substituting $q^{-1}$
for $q$. Another example is given as follows: let $k=\FF_{q}$ and
$k_{r}:=\FF_{q^{r}}$. For a $k$-variety $X$, the function 
\[
r\mapsto\sharp^{r}X:=\sharp X(k_{r})
\]
is admissible from the Grothendieck-Lefschetz trace formula.
\begin{lem}
The dual function $\DD f$ does not depend on the choice of the expression
of $f(r)$ as $\frac{\sum_{i=1}^{l}n_{i}\alpha_{i}^{r}}{q^{cr}-1}$. \end{lem}
\begin{proof}
Let 
\[
\frac{\sum_{j=1}^{l'}m_{j}\beta_{j}^{r}}{q^{dr}-1}
\]
be another such expression of $f(r)$. For every $r\in\NN$, we have
\[
(q^{dr}-1)\cdot\sum_{i=1}^{l}n_{i}\alpha_{i}^{r}=(q^{cr}-1)\cdot\sum_{j=1}^{l'}m_{j}\beta_{j}^{r}.
\]
It suffices to show that for every $r\in\NN$, 
\[
(q^{-dr}-1)\cdot\sum_{i=1}^{l}n_{i}\alpha_{i}^{-r}=(q^{-cr}-1)\cdot\sum_{j=1}^{l'}m_{j}\beta_{j}^{-r}.
\]
This follows from the following claim:
\begin{claim}
Suppose that for every $r$, we have 
\[
\sum_{i=1}^{l}n_{i}\alpha_{i}^{r}=\sum_{j=1}^{l'}m_{j}\beta_{j}^{r}.
\]
Suppose also that all of $\alpha_{i}$, $n_{i}$, $\beta_{j}$ and
$m_{j}$ are nonzero, and that $\alpha_{1},\dots,\alpha_{l}$ are
distinct, and so are $\beta_{1},\dots,\beta_{l'}$. Then we have $l=l'$,
$\alpha_{i}=\beta_{i}$ and $n_{i}=m_{i}$ up to permutation. 
\end{claim}
In turn, this claim follows from the following one:
\begin{claim}
Suppose that $\alpha_{1},\dots,\alpha_{l}$ are distinct nonzero complex
numbers and that for every $r$, 
\[
\sum_{i=1}^{l}n_{i}\alpha_{i}^{r}=0.
\]
Then $ $$n_{i}=0$ for every $i$. 
\end{claim}
The last claim follows from the regularity of the Vandermonde matrix
\[
(\alpha_{i}^{r})_{1\le i\le l,\,0\le r\le l-1}.
\]
.
\end{proof}
For a smooth proper $k$-variety $X$ of pure dimension $d$, the
Poincaré duality for the $l$-adic cohomology shows that the function
$r\mapsto\sharp^{r}X$ satisfies 
\[
\sharp^{r}X=q^{dr}\cdot\DD(\sharp^{r}X).
\]

\section{Stringy point counts\label{sec:stringy}}

\subsection{Stringy point counts}

Let $K$ be a local field, that is, a finite extension of either $\QQ_{p}$
or $\FF_{p}((t))$. We denote its residue field by $k$ and its integer
ring by $\cO_{K}$. An \emph{$\cO_{K}$-variety} means an integral
separated flat $\cO_{K}$-scheme of finite type such that there exists
an open dense subscheme $U\subset X$ which is smooth over $\cO_{K}$.
For an $\cO_{K}$-variety $X$, we denote by $X_{k}$ the closed fiber
$X\otimes_{\cO_{K}}k$ and by $X_{K}$ the generic fiber $X\otimes_{\cO_{K}}K$.
For a constructible subset $C\subset X_{k}$, we let $X(\cO_{K})_{C}$
to the set of $\cO_{K}$-points $\Spec\cO_{K}\to X$ sending the closed
point  into $C$. 

Suppose now that $X$ is normal. From \cite[Definition 1.5]{MR3057950},
$X$ has a canonical sheaf $\omega_{X/\cO_{K}}$. If $d$ is the relative
dimension of $X$ over $\cO_{K}$, then $\omega_{X/\cO_{K}}$ coincides
with $\Omega_{X/\cO_{K}}^{d}=\bigwedge^{d}\Omega_{X/\cO_{K}}$ on
the $\cO_{K}$-smooth locus of $X$. The canonical divisor $K_{X}=K_{X/\cO_{K}}$
is defined as a divisor such that $\omega_{X/\cO_{K}}=\cO_{X}(K_{X})$,
which is determined up to linear equivalence. Let us suppose also
that $X$ is $\QQ$-Gorenstein, that is, $K_{X}$ is $\QQ$-Cartier.
Then, for a constructible subset $C\subset X_{k}$, we can define
the \emph{stringy point count along $C$, 
\[
\sharp_{\st}(X)_{C}\in\RR_{\ge0}\cup\{\infty\},
\]
}as a certain $p$-adic volume of $X(\cO_{K})_{C}$\emph{.} For details,
see \cite{wild-p-adic}. However, what is important for our purpose
is not the definition but the explicit formula in the next subsection.

For $r\in\NN$, let $K_{r}$ be the unramified extension of $K$ of
degree $r$ and let $k_{r}$ be its residue field. We then define
$\sharp_{\st}^{r}(X)_{C}$ to be the stringy point count $X\otimes_{\cO_{K}}\cO_{K_{r}}$
along the preimage of $C$ in $(X\otimes_{\cO_{K}}\cO_{K_{r}})\otimes_{\cO_{K_{r}}}k_{r}=X\otimes_{\cO_{K}}k_{r}$.
We often regard $\sharp_{\st}^{r}(X)_{C}$ as a function in $r$.
When $X$ is smooth over $\cO_{K}$, then $\sharp_{\st}^{r}(X)_{C}$
is equal to the number of $k_{r}$-points of $X$ contained in $C$.

\subsection{An explicit formula and the Poincaré duality}

Let $X$ be a normal $\QQ$-Gorenstein variety over either $\cO_{K}$
or $k$. For a proper birational morphism $f:Y\to X$ with $Y$ normal,
the relative canonical divisor 
\[
K_{Y/X}:=K_{Y}-f^{*}K_{X}
\]
is uniquely determined as a $\QQ$-divisor on $Y$ supported in the
exceptional locus of $f$. We are especially interested in the case where $Y$
is regular and $K_{Y/X}$ is simple normal crossing. We define simple
normal crossing divisors on regular $\cO_{K}$-varieties as follows.
\begin{defn}
Let $Y$ be a regular variety over  $\cO_{K}$ and $E\subset Y$
a reduced closed subscheme of pure codimension one. We call $E$ a
\emph{simple normal crossing divisor }if 
\begin{itemize}
\item for every irreducible component $E_{0}$ of $E$ and for every closed
point $x\in E_{0}$ where $Y$ is smooth over $\cO_{K}$, the completion
of $E_{0}$ at $x$ is irreducible, and
\item for every closed point $x\in Y$ where $Y$ is smooth over $\cO_{K}$,
the support of $E$ is, in a certain Zariski neighborhood, defined
by a product $y_{i_{1}}\dots y_{i_{m}}$ for a regular system of parameters
$y_{0}=\varpi,y_{1}\dots,y_{d}\in\hat{\cO}_{Y,y}$ with $\varpi$
a uniformizer of $K$ and $0\le i_{1}<\cdot<i_{m}\le d$. 
\end{itemize}
We then call a $\QQ$-divisor on $Y$ \emph{simple normal crossing
}if so is its support.
\end{defn}
The reason why we look only at closed points where $Y$ is $\cO_{K}$-smooth
is that there is no $\cO_{K}$-point passing through a closed point
where $Y$ is not $\cO_{K}$-smooth, and hence such a point is irrelevant
to stringy point counts.
\begin{defn}
Let $f:Y\to X$ be a proper birational morphisms of varieties over
either $\cO_{K}$ or $k$ such that $X$ is normal and $\QQ$-Gorenstein.
\begin{itemize}
\item We call $f$ a \emph{resolution }if $Y$ is regular.
\item We call $f$ a \emph{WL (weak log) resolution }if $f$ is a resolution
and the support of $K_{Y/X}$ is simple normal crossing.
\item When $f$ is defined over $\cO_{K}$, we call $f$ an \emph{SWL (simultaneous
weak log)} \emph{resolution} if $f$ is a WL resolution and $Y$ is
smooth over $\cO_{K}$ or $k$. 
\end{itemize}
\end{defn}
\begin{rem}
The word \emph{weak }indicates that we look only at the support of
$K_{Y/X}$ rather than the whole exceptional locus. In particular,
every resolution $Y\to X$ with $K_{Y/X}=0$ is a WL resolution, whatever
the exceptional locus is.
\begin{rem}
We cannot usually expect that an SWL resolution exists, unless $X$
is thought of as an \emph{equisingular }family over $\cO_{K}$ in
some sense. If $\cO_{K}=k[[t]]$ and if $X=X_{0}\otimes_{k}k[[t]]$
for some $k$-variety $X_{0}$, then $X$ would be equisingular in
any reasonable sense. In this situation, an SWL resolution of $X$
exists if and only if a WL resolution of $X_{0}$ exists. 
\end{rem}
\end{rem}
For a normal $\QQ$-Gorenstein $\cO_{K}$-variety $X$ and a WL resolution
$f:Y\to X$, we can uniquely decompose $K_{Y/X}$ as 
\begin{equation}
K_{Y/X}=\sum_{h=1}^{l}a_{h}A_{h}+\sum_{i=1}^{m}b_{i}B_{i}+\sum_{j=1}^{n}c_{j}C_{j},\label{discrep}
\end{equation}
where all the $A_{h},B_{i},C_{j}$ are prime divisors on $Y$, the
coefficients $a_{h},b_{i},c_{j}$ are nonzero rational numbers, every
$A_{h}$ is contained in $X_{k}$ and $Y$ is $\cO_{K}$-smooth at
the generic point of $A_{h}$, every $B_{i}$ is contained in $X_{k}$
and $Y$ is \emph{not }$\cO_{K}$-smooth at the generic point of $B_{i}$,
and every $C_{j}$ dominates $\Spec\cO_{K}$. We denote by $A_{h}^{\circ}$
the locus in $A_{h}$ where $Y$ is $\cO_{K}$-smooth. For a subset
$J\subset\{1,\dots,n\}$, we define
\[
C_{J}^{\circ}:=\bigcap_{j\in J}C_{j}\setminus\bigcup_{j\notin J}C_{j}.
\]
For a constructible subset $W$ of a $k$-variety $Z$, let $\sharp^{r}(W)$ denote the number of
$k_{r}$-points of $W$.

\begin{prop}[An explicit formula, \cite{wild-p-adic}]
\label{prop:explicit-formula-stringy}Let $f:Y\to X$ be a WL resolution
of a normal $\QQ$-Gorenstein $\cO_{K}$-variety $X$ and write $K_{Y/X}$
as in (\ref{discrep}). Let $Y_{\sm}$ be the $\cO_{K}$-smooth locus
of $Y$. 
\begin{enumerate}
\item \label{enu: log terminal}We have $\sharp_{\st}^{r}(X)_{C}<\infty$
for all $r\in\NN$ if and only if for every $j$ such that $C_{j}\cap f^{-1}(C)\cap Y_{\sm}\ne\emptyset$,
$c_{j}>-1$.
\item If the two equivalent conditions of (\ref{enu: log terminal}) hold,
then 
\[
\sharp_{\st}^{r}(X)_{C}=\sum_{h=1}^{l}q^{-ra_{h}}\sum_{J\subset\{1,\dots,n\}}\sharp^{r}(f^{-1}(C)\cap A_{h}^{\circ}\cap C_{J}^{\circ})\prod_{j\in J}\frac{q^{r}-1}{q^{r(1+c_{j})}-1}.
\]

\end{enumerate}
\end{prop}
\begin{proof}
In \cite{wild-p-adic} it was proved that $\sharp_{\st}(X)_{C}$ is
equal to $\sharp_{\st}(Y,-K_{Y/X})_{f^{-1}(C)}$, the stringy point
count of the pair $(Y,-K_{Y/X})$ along $f^{-1}(C)$. The proposition
follows from the explicit formula for the stringy point count of a
pair with a simple normal crossing divisor in the same paper.
\end{proof}

In particular, if $K$ has characteristic zero and if the generic fiber $X_K$ has only log terminal singularities (for instance, quotient singularities), then $\sharp^r (X)_C<\infty $  for every $r$ and $C$. 
The following is a direct consequence of the proposition. 
\begin{cor}
Suppose that $X$ is $\QQ$-Gorenstein, that there exists a WL resolution
$f:Y\to X$, and that for every $r>0$, $\sharp_{\st}^{r}(X)_{C}<\infty$.
Then the function $\sharp_{\st}^{r}(X)_{C}$ in the variable $r$
is admissible.
\end{cor}
The following result was proved by Batyrev--Dais \cite{MR1404917} and Batyrev \cite{MR1672108} in a
slightly different setting. 
\begin{cor}[The Poincaré duality]
\label{cor:Poincare dual}Let $X$ be a $d$-dimensional proper normal
$\QQ$-Gorenstein $\cO_{K}$-variety. Suppose that there exists an
SWL resolution $f:Y\to X$ and that for every $r\in\NN$, $\sharp_{\st}^{r}(X)<\infty$.
We then have 
\[
\sharp_{\st}^{r}(X)=\DD(\sharp_{\st}^{r}(X))\cdot q^{dr}.
\]
\end{cor}
\begin{proof}
We follow arguments in the proof of \cite[Theorem 3.7]{MR1672108}.
Since $f$ is an SWL resolution, if we write $K_{Y/X}$ as in (\ref{discrep}),
then the terms $\sum a_{h}A_{h}$ and $\sum b_{i}B_{i}$ do not appear.
Therefore Proposition \ref{prop:explicit-formula-stringy}$ $ reads
\[
\sharp_{\st}^{r}(X)=\sum_{J\subset\{1,\dots,n\}}\sharp^{r}(C_{J}^{\circ})\prod_{j\in J}\frac{q^{r}-1}{q^{r(1+c_{j})}-1}.
\]
Putting 
\[
C_{J}:=\bigcap_{j\in J}C_{j},
\]
we have $C_{J}:=\bigsqcup_{J'\supset J}C_{J'}^{\circ}$ and $\sharp^{r}(C_{J})=\sum_{J'\supset J}\sharp^{r}(C_{J'}^{\circ}).$
From this and the inclusion-exclusion principle (or the Möbius inversion
formula, for instance, see \cite{MR1442260}), we deduce
\[
\sharp^{r}(C_{J}^{\circ})=\sum_{J'\supset J}(-1)^{\sharp J'-\sharp J}\sharp^{r}(C_{J'}).
\]
Hence
\begin{align*}
\sharp_{\st}^{r}(X) & =\sum_{J\subset\{1,\dots,n\}}\left(\sum_{J'\supset J}(-1)^{\sharp J'-\sharp J}\sharp^{r}(C_{J'})\right)\prod_{j\in J}\frac{q^{r}-1}{q^{r(1+c_{j})}-1}\\
 & =\sum_{J\subset\{1,\dots,n\}}\sharp^{r}(C_{J})\prod_{j\in J}\left(\frac{q^{r}-1}{q^{r(1+c_{j})}-1}-1\right)\\
 & =\sum_{J\subset\{1,\dots,n\}}\sharp^{r}(C_{J})\prod_{j\in J}\left(\frac{q^{r}-q^{r(1+c_{j})}}{q^{r(1+c_{j})}-1}\right).
\end{align*}
Now we have 
\[
\DD\left(\frac{q^{r}-q^{r(1+c_{j})}}{q^{r(1+c_{j})}-1}\right)=\frac{q^{-r}-q^{r(-1-c_{j})}}{q^{r(-1-c_{j})}-1}=\frac{q^{r}-q^{r(1+c_{j})}}{q^{r(1+c_{j})}-1}\cdot q^{-r}.
\]
Since $\sharp^{r}(C_{J})=\sharp^{r}(C_{J}\otimes_{K}k)$ and $C_{J}\otimes_{K}k$
are proper smooth $k$-varieties of pure dimension $d-\sharp J$,
the Poincaré duality for the $l$-adic cohomology gives
\[
\DD(\sharp^{r}(C_{J}))=\sharp^{r}(C_{J})\cdot q^{r(\sharp J-d)}.
\]
Since $\DD:\AF\to\AF$ is a ring homomorphism, the corollary follows.
\end{proof}

\subsection{Set-theoretically free $\GG_{m}$-actions\label{sub: free-actions}}
\begin{defn}
An action of the group scheme $\GG_{m}=\GG_{m,k}=\Spec k[t,t^{-1}]$
on a $k$-variety $Z$ is said to be \emph{set-theoretically free
}if for every field extension $k'/k$ and every point $x\in Z(k')$,
the stabilizer subgroup scheme $\Stab(x)\subset\GG_{m,k'}$ consists
of a single point.
\end{defn}
A set-theoretically free $\GG_{m}$-action naturally appears as the
$\GG_{m}$-action on the quotient variety $(\AA_{k}^{d}/\Gamma)\setminus\{o\}$
with the origin removed for a linear action of a finite group $\Gamma$
on an affine space $\AA_{k}^{d}$. In general, the $\GG_{m}$-action may have
non-reduced stabilizer subgroups and not be free. See \cite{MR3230848}
for such an example. If $X$ is the quotient variety $Z/\GG_{m}$
for a set-theoretically free $\GG_{m}$-action on a smooth variety
$Z$ and if $X$ is proper, then the function $\sharp^{r}X$ satisfies
the same Poincaré duality as a smooth proper variety does. 
Note that $X$ is not necessarily smooth because of non-reduced stabilizers. 
To prove the duality, we need to use Artin stacks. 

Based on works of Laszlo and Olsson \cite{Olsson:2008tv,MR2434692,MR2434693},
Sun \cite{MR2950161} defined étale cohomology groups $H^{i}(\cX\otimes_{k}\overline{k},\overline{\QQ}_{l})$
and $H_{c}^{i}(\cX\otimes_{k}\overline{k},\overline{\QQ}_{l})$ for
Artin stacks $\cX$ of finite type over $k$ and a prime number $l\ne p$.
For $r\in\NN$, we define 
\[
\sharp^{r}(\cX):=\sum_{i}(-1)^{i}\mathrm{Tr}(F^r | H_{c}^{i}(\cX\otimes_{k}\bar{k},\overline{\QQ}_{l}))
\]
with $F^r$  the $r$-iterated Frobenius action. 
This is a generalization of $\sharp^{r}X=\sharp X(k_{r})$ for a $k$-variety
$X$.
\begin{prop}
\label{prop: Artin et char}Let $\cX$ be an Artin stack of finite type over $k$
with finite diagonal, and $\overline{\cX}$ its coarse moduli space.
Then we have
\[
\sharp^{r}(\cX)=\sharp^{r}(\overline{\cX}).
\]
Moreover, if $\cX$ is smooth and proper of pure dimension $d$ over
$k$, then
\[
\sharp^{r}(\cX)=\DD(\sharp^{r}(\cX))\cdot q^{rd}.
\]
\end{prop}
\begin{proof}
In the proof of \cite[Proposition 7.3.2]{MR2950161}, Sun proved that
\[
H_{c}^{i}(\cX\otimes_{k}\bar{k},\overline{\QQ}_{l})=H_{c}^{i}(\overline{\cX}\otimes_{k}\bar{k},\overline{\QQ}_{l})
\]
and that if $\cX$ is smooth and proper, then we have the Poincaré
duality,
\[
H^{i}(\cX\otimes_{k}\bar{k},\overline{\QQ}_{l})^{\vee}=H^{2d-i}(\cX\otimes_{k}\bar{k},\overline{\QQ}_{l})\otimes\overline{\QQ}_{l}(d).
\]
The proposition is a direct consequence of these results.
\end{proof}
Let $Z$ be a $k$-variety endowed with a set-theoretically free $\GG_{m}$-action.
Then the quotient stack $\cW:=[Z/\GG_{m}]$ is an Artin stack with
finite diagonal and hence admits a coarse moduli space, which we write
as $W=Z/\GG_{m}$.
\begin{prop}
We have 
\[
\sharp^{r}(\cW)=\sharp^{r}(W)=\frac{\sharp^{r}(Z)}{q^{r}-1}.
\]
\end{prop}
\begin{proof}
The left equality was proved in the last proposition. We show the
right equality. For every field extension $k'/k$ and every point
$x\in Z(k')$, we have 
\[
\Stab(x)=\Spec k'[t,t^{-1}]/(t^{p^{a}}-1)
\]
for some non-negative integer $a$. Let $U\subset Z$ be the locus
where this number $a$ takes the minimum value, and let $\overline{U}\subset W$
be its image. The algebraic spaces $U$ and $\overline{U}$ are open
dense subspaces of $Z$ and $W$ respectively. If we put 
\[
H:=\Spec k[t,t^{-1}]/(t^{p^{a}}-1),
\]
then the given set-theoretically free $\GG_{m}$-action on $U$ induces
a free action of the quotient group scheme $\GG_{m}/H$, which is
isomorphic to $\GG_{m}$, on $U$. This action makes the projection
$U\to\overline{U}$ a $\GG_{m}$-torsor. From Hilbert's Theorem 90
\cite[page 124]{MR559531}, every $\GG_{m}$-torsor is Zariski locally
trivial, hence 
\[
\sharp^{r}(U)=(q^{r}-1)\cdot\sharp^{r}(\overline{U}).
\]
It is now easy to show $\sharp^{r}(Z)=(q^{r}-1)\cdot\sharp^{r}(W)$
by induction. \end{proof}
\begin{defn}
Let $X$ be an $\cO_{K}$-variety with a $\GG_{m,\cO_{K}}$-action.
We suppose that the induced $\GG_{m,k}$-action on $X_{k}$ is set-theoretically
free. For the quotient stack $[X/\GG_{m,\cO_{K}}]$, we put
\[
\sharp_{\st}^{r}([X/\GG_{m,\cO_{K}}]):=\frac{\sharp_{\st}^{r}(X)}{q^{r}-1}.
\]

\begin{defn}
\label{def: ESWL res}For a normal $\QQ$-Gorenstein $\cO_{K}$-variety
$X$ endowed with a $\GG_{m,\cO_{K}}$-action, we call a resolution
$f:Y\to X$ an \emph{ESWL (equivariant simultaneous weak log) resolution}
if $f$ is an SWL resolution and $\GG_{m,\cO_{K}}$-equivariant (that
is, $Y$ also has a $\GG_{m,\cO_{K}}$-action so that $f$ is $\GG_{m,\cO_{K}}$-equivariant).
\end{defn}
\end{defn}
\begin{prop}
\label{prop:DualPX}Let $X$ be a normal $\QQ$-Gorenstein $\cO_{K}$-variety
$X$ endowed with a $\GG_{m,\cO_{K}}$-action. Suppose that the induced
action of $\GG_{m,k}$ on $X_{k}$ is set-theoretically free and the
quotient stack $[X_{k}/\GG_{m,k}]$ is proper over $k$. Suppose also
that there exists an ESWL resolution $f:Y\to X$, and that $\sharp_{\st}^{r}(X)<\infty$
for all $r$. Then 
\[
\sharp_{\st}^{r}([X/\GG_{m,\cO_{K}}])=\DD(\sharp_{\st}^{r}([X/\GG_{m,\cO_{K}}]))\cdot q^{r(d-1)}
\]
with $d$ the dimension of $X$ over $\cO_{K}$.\end{prop}
\begin{proof}
From Lemma \ref{prop:explicit-formula-stringy}, if we write $K_{Y/X}=\sum_{j=1}^{n}c_{j}C_{j}$,
then we obtain 
\begin{align*}
\sharp_{\st}^{r}([X/\GG_{m,\cO_{K}}]) & =\sum_{J\subset\{1,\dots,n\}}\frac{\sharp^{r}(C_{J}^{\circ})}{q^{r}-1}\prod_{i\in J}\frac{q^{r}-1}{q^{r(1+c_{j})}-1}\\
 & =\sum_{J\subset\{1,\dots,n\}}\frac{\sharp^{r}(C_{J})}{q^r-1}\prod_{j\in J}\left(\frac{q^{r}-q^{r(1+c_{j})}}{q^{r(1+c_{j})}-1}\right)
\end{align*}
in the same way as in the proof of Corollary \ref{cor:Poincare dual}.
Each $C_{J}\cap Y_{k}$ is stable under the $\GG_{m,k}$-action and
the action on $C_{J}\cap Y_{k}$ is set-theoretically free. Hence
$[(C_{J}\cap Y_{k})/\GG_{m,k}]$ is a smooth proper Artin $k$-stack
of dimension $d-1-\sharp J$ with finite diagonal. From Proposition
\ref{prop: Artin et char}, 
\[
\frac{\sharp^{r}(C_{J})}{q^{r}-1}=\DD\left(\frac{\sharp^{r}(C_{J})}{q^{r}-1}\right)\cdot q^{r(d-1-\sharp J)}.
\]
We can show the rest of the proof as in the proof of Corollary \ref{cor:Poincare dual}.
\end{proof}

\section{Total masses of local Galois representations\label{sec:Total-masses}}

\subsection{Total masses}

Let $K$ be a local field and $G_{K}=\Gal(K^{\sep}/K)$ its absolute Galois group.
For a finite group $\Gamma$, we define $S_{K,\Gamma}$ to be the
set of continuous homomorphisms $G_{K}\to\Gamma$. To each representation
$\tau:\Gamma\to\GL d{\cO_{K}},$ we can associate, among others, three
functions $S_{K,\Gamma}\to\RR$ denoted by $\ba_{\tau}$, $\bv_{\tau}$
and $\bw_{\tau}$.

The first one $\ba_{\tau}$ is called the \emph{Artin conductor}.
For $\rho\in S_{K,\Gamma}$, let $H$ be the image of $\rho$ and
$L/K$ the associated Galois extension having $H$ as the Galois group.
The Galois group $H$ has  the filtration by ramification subgroups with lower numbering
\[
H\supset H_{0}\supset H_{1}\supset\cdots,
\]
(see \cite{MR554237}). We define
\[
\ba_{\tau}(\rho):=\sum_{i=0}^{\infty}\frac{1}{(H_{0}:H_{i})}\codim\,(K^{d})^{H_{i}}.
\]

The second function $\bv_{\tau}$ is defined as follows. For $\rho\in S_{K,\Gamma}$,
let $\Spec M\to\Spec K$ be the corresponding étale $\Gamma$-torsor
and $\cO_{M}$ the integer ring of $M$. Note that the spectrum of
the extension $L$ mentioned above is a connected component of $\Spec M$.
Then the free $\cO_{M}$-module $\cO_{M}^{\oplus d}$ has two (left)
$\Gamma$-actions: firstly the action induced from the map
\[
\Gamma\xrightarrow{\tau}\GL d{\cO_{K}}\subset\GL d{\cO_{M}}
\]
and secondly the diagonal action induced from the given $\Gamma$-action
on $\cO_{M}$. We define the associated \emph{tuning submodule} $\Xi\subset\cO_{M}^{\otimes d}$ to be the
subset of those element on which the two actions coincide. It turns
out that $\Xi$ is a free $\cO_{K}$-module of rank $d$. We define
\[
\bv_{\tau}(\rho):=\frac{1}{\sharp\Gamma}\cdot\length\left(\frac{\cO_{M}^{\oplus d}}{\cO_{M}\cdot\Xi}\right).
\]
Note that we may also use $\cO_{L}$ and the subgroup $H=\Im(\rho)$
in the definition of $\bv_{\tau}$ instead of $\cO_{M}$ and $\Gamma$. 

The last function $\bw_{\tau}$ is called the \emph{weight function.}
Let $\rho\in S_{K,\Gamma}$, $M$ and $\Xi \subset \cO_M^{\oplus d}$ be as above.
We identify $\Xi$ with the subgroup of $\Gamma$-equivariant maps
of $\Hom _{\cO_K}(\cO_K^{\oplus d}, \cO_M)=\cO_M^{\oplus d}$.
For  an $\cO_{K}$-basis $\phi_{1},\dots,\phi_{d}$  of $\Xi$, we
define an $\cO_{K}$-algebra homomorphism $u^{*}:\cO_{K}[x_{1},\dots,x_{d}]\to\cO_{M}[y_{1},\dots,y_{d}]$
by $u^{*}(x_{i})=\sum_{j=1}^{n}\phi_{j}(x_{i})y_{j}$. Let $u:\AA_{\cO_{M}}^{d}\to\AA_{\cO_{K}}^{d}$
be the corresponding morphism of schemes and $o:\Spec k\hookrightarrow\AA_{\cO_{K}}^{d}$
the $k$-point at the origin. We define
\[
\bw_{\tau}(\rho):=\dim u^{-1}(o)-\bv_{\tau}(\rho).
\]

\begin{rem}
\label{rem: difference defs}Our definition of $\bw_{\tau}$ follows
the one in \cite{wild-p-adic} and is slightly different from the
one in the preceding paper \cite{Wood-Yasuda-I}. Let us consider
the left $\Gamma$-action on $\AA_{\cO_{K}}^{d}$ induced from the
one on $\cO_{K}[x_{1},\dots,x_{d}]$. In precise, an element $\gamma\in\Gamma$
gives the automorphism of $\AA_{\cO_{K}}^{d}$corresponding to the
automorphism on $\cO_{K}[x_{1},\dots,x_{d}]$ given by the inverse
$\gamma^{-1}$ of $\gamma$. In \cite{Wood-Yasuda-I} we defined 
\[
\bw_{\tau}(\rho):=\codim\left((\AA_{k}^{d})^{H_{0}},\AA_{k}^{d}\right)-\bv_{\tau}(\rho),
\]
where $H_0$ is the $0$-th ramification subgroup of $H$, that is, the inertia subgroup.
However, as proved in \cite{wild-p-adic}, the two definitions coincide
in the following three cases which we are mainly concerned with.
\begin{enumerate}
\item The order of the given group $\Gamma$ is invertible in $k$.
\item The given representation $\tau$ is a permutation representation. 
\item The given local field $K$ has positive characteristic, and hence
$K=k((t))$, and the image of $\tau$ is contained in $\GL dk$. 
\end{enumerate}
\end{rem}
\begin{defn}
For a function $c:S_{K,\Gamma}\to\RR$, we define the \emph{total
mass }of $(K,\Gamma,c)$ as\emph{ }
\[
M(K,\Gamma,c):=\frac{1}{\sharp\Gamma}\cdot\sum_{\rho:\in S_{K,\Gamma}}q^{-c(\rho)}\in\RR_{\ge0}\cup\{\infty\}.
\]

\end{defn}
For $r\in\NN$, let $K_{r}/K$ be the unramified extension of degree
$r$. We continue to denote the induced representations $\Gamma\xrightarrow{\tau}\GL n{\cO_{K}}\hookrightarrow\GL n{\cO_{K_{r}}}$
by $\tau$. Then we can consider total masses $M(K_{r},\Gamma,\ba_{\tau})$,
$M(K_{r},\Gamma,\bv_{\tau})$ and $M(K_{r},\Gamma,-\bw_{\tau})$,
and regard them as functions in $r$. 

Kedlaya studied total masses for $c=\ba_{\tau}$ and used it to interpret
and generalize Bhargava's mass formula \cite{MR2354798}. Wood \cite{MR2411405}
studied ones for different choices of $c$. On the context of the
wild McKay correspondence, explained below, functions $\bv_{\tau}$
and $\bw_{\tau}$ are concerned. For permutation representations,
functions $\ba_{\tau}$ and $\bv_{\tau}$ are related as follows.
\begin{lem}[\cite{Wood-Yasuda-I}]
\label{lem:permutation}If $\tau$ is a permutation representation,
then
\[
2\bv_{\tau}=\bv_{\tau\oplus\tau}=\ba_{\tau}.
\]

\end{lem}

\subsection{The wild McKay correspondence}

A representation $\tau:\Gamma\to\GL d{\cO_{K}}$ defines $\Gamma$-actions
on the polynomial ring $\cO_{K}[x_{1},\dots,x_{d}]$ as above. Let
$X$ be the quotient variety $\AA_{\cO_{K}}^{d}/\Gamma=\Spec\cO_{K}[x_{1},\dots,x_{d}]^{\Gamma}$,
which is a normal $\QQ$-Gorenstein $\cO_{K}$-variety. The following
theorem, which we call the wild McKay correspondence, connect total masses and stringy point counts.
This was proved in \cite{Wood-Yasuda-I,MR3230848} for special cases and in \cite{wild-p-adic} in full generality.

\begin{thm}[The wild McKay correspondence]
\label{thm:McKay realized}Suppose that the representation $\tau$
is faithful and the quotient morphism $\AA_{\cO_{K}}^{d}\to X$ is
étale in codimension one. We have 
\begin{gather*}
\sharp_{\st}^{r}(X)=M(K_{r},\Gamma,\bv_{\tau})\cdot q^{dr} \text{ and}\\
\sharp_{\st}^{r}(X)_{o}=M(K_{r},\Gamma,-\bw_{\tau}).
\end{gather*}

\end{thm}

The following corollary is a direct consequence of this theorem and
Proposition \ref{prop:explicit-formula-stringy}.
\begin{cor}
If $X$ admits a WL resolution and if $M(K_{r},\Gamma,\bv_{\tau})<\infty$
(resp. $M(K_{r},\Gamma,-\bw_{\tau})<\infty$) for every $r\in\NN$,
then the function $M(K_{r},\Gamma,\bv_{\tau})$ (resp. $M(K_{r},\Gamma,-\bw_{\tau})$)
is admissible.
\end{cor}

\subsection{Dualities of total masses}

To discuss dualities, let us now assume that both functions $M(K_{r},\Gamma,\bv_{\tau})$
and $M(K_{r},\Gamma,-\bw_{\tau})$ have finite values for all $r$
and are admissible. The quotient variety $X$ associated to a faithful
representation $\tau:\Gamma\to\GL d{\cO_{K}}$ has a natural $\GG_{m,\cO_{K}}$-action.
Let $X^{*}$ be the complement of the zero section $\Spec\cO_{K}\hookrightarrow X$.
The $\GG_{m,k}$-action on $X_{k}^{*}$ is set-theoretically free.
From definition, we have
\[
\sharp_{\st}^{r}([X^{*}/\GG_{m,\cO_{K}}])=\frac{\sharp_{\st}^{r}(X)-\sharp_{\st}^{r}(X)_{o}}{q^{r}-1}.
\]
If there exists an ESWL resolution $Y\to X^{*}$, then from Proposition
\ref{prop:DualPX}, we have
\begin{equation}\label{estdual}
\frac{\sharp_{\st}^{r}(X)-\sharp_{\st}^{r}(X)_{o}}{q^{r}-1}=\DD\left(\frac{\sharp_{\st}^{r}(X)-\sharp_{\st}^{r}(X)_{o}}{q^{r}-1}\right)\cdot q^{r(d-1)}.
\end{equation}

\begin{rem}
When  $\Gamma =\ZZ/p\ZZ \subset \GL d k$ with $k$  a perfect field of characteristic $p>0$,
  the motivic version of \eqref{estdual} was checked in \cite[Prop.\ 6.36]{MR3230848},  by using 
  not resolution but an explicit formula for motivic total masses. 
\end{rem}

If $\AA_{\cO_{K}}^{d}\to X$ is étale in codimension one, then 
the wild McKay correspondence, Theorem \ref{thm:McKay realized}, shows that the last equality is equivalent to: 
\begin{multline}
M(K_{r},\Gamma,\bv_{\tau})\cdot q^{rd}-M(K_{r},\Gamma,-\bw_{\tau})\\
=\DD(M(K_{r},\Gamma,-\bw_{\tau}))\cdot q^{rd}-\DD(M(K_{r},\Gamma,\bv_{\tau})).\label{e2}
\end{multline}
We call this the \emph{weak duality}. 
The equality may be regarded as a constraint imposed by the existence
of an ESWL resolution of $X^{*}$. If this duality fails for some
representation $\tau$, then the associated $X^{*}$ would have no
ESWL resolution. In particular, it is interesting to ask the following
question.
\begin{question}\label{Q:positive_char_weak_dual}
Suppose that $K$ has positive characteristic so that $K=k((t))$,
that $\tau:\Gamma\to\GL d{\cO_{K}}$ is a faithful representation
factoring through $\GL nk$ and that the quotient morphism $\AA_{\cO_{K}}^{d}\to\AA_{\cO_{K}}^{d}/\Gamma$
is étale in codimension one. Then, does the weak duality always
hold?
\end{question}
If this has the negative answer and some representation $\tau$ does
not satisfy (\ref{e2}), then the associated $k$-variety $X_{k}=\AA_{k}^{d}/\Gamma$
does not admit any $\GG_{m,k}$-equivariant WL resolution. 

In some examples,  even a stronger
equality,
\begin{equation}
\DD(M(K_{r},\Gamma,-\bw_{\tau}))=M(K_{r},\Gamma,\bv_{\tau}),\label{eqmd}
\end{equation}
 holds: we call this the \emph{strong duality}.
 A special case of this  was the duality involving
Bhargava's formula and the Hilbert scheme of points observed in \cite{Wood-Yasuda-I} (see Section \ref{subsec:Bhargava}).
Actually, as far as the authors know, even the strong duality holds in all examples satisfying the assumption as in Question \ref{Q:positive_char_weak_dual}.  It is thus natural to ask:

\begin{question}\label{Q:positive_char_strong_dual}
With the same assumption as in Question \ref{Q:positive_char_weak_dual},  does the strong duality  always
hold?
\end{question}

\section{Examples\label{sec:Examples}}

\subsection{The tame case}

We first consider the tame case.
Suppose  $p\nmid \sharp \Gamma$. 
Let $\tau : \Gamma \to \GL d{\cO_K}$ be a representation. Then $S_{K,\Gamma}$ is a finite set and every $\rho \in S_{K,\Gamma}$ factors through the Galois group $G_K^\tame$ of the maximal tamely ramified extension
$ K^\tame /K$. It is topologically generated by two elements $a$ and $b$ with one relation 
$bab^{-1}=a^q$ (see \cite[page 410]{MR2392026}). Here $a$ is the topological generator of the inertia subgroup of $G_K^\tame$.
Therefore $S_{K,\Gamma}$ is in a one-to-one correspondence with
\[
\{  (g,h) \in \Gamma^2 \mid hgh^{-1}=g^q \}.
\]
The last set admits the involution $(g,h)\mapsto (g^{-1},h)$.
Let $\iota$ be the corresponding involution of $S_{K,\Gamma}$.

\begin{lem}
For $\rho\in S_{K,\Gamma}$, we have $\bv_\tau (\iota (\rho) )  = \bw _\tau(\rho) $.
\end{lem}

\begin{proof}
Let $K'$ be the completion of the maximal unramified extension of $K$. 
Both $\rho$ and $\iota(\rho)$ define the same cyclic extension $L/K'$  of order $l:=\ord(g)$
with $(g,h)\in \Gamma^2$ corresponding to $\rho$. 
We can regard the element $a \in G^\tame_K$ also as a generator of $\Gal(L/K')$.
Let $\varpi\in L$ be a uniformizer and $\zeta \in  K'$ an $l$-th root of unity such that
$a(\varpi) = \zeta \varpi$. 

If $\zeta^{a_1},\dots,\zeta^{a_d}$ with $0 \le a_i < l$ are the eigenvalues of $g \in \Gamma\subset \GL d { K'}$,
then, from \cite[Lem.\ 4.3]{Wood-Yasuda-I}, we have 
\[
 \bv_\tau (\rho) = \frac{1}{l} \sum _{i=1}^l a_i . 
\]

In the tame case, our definition of $\bw_\tau $ coincides with the one in \cite{Wood-Yasuda-I} 
(see Remark \ref{rem: difference defs}).
Therefore
\begin{align*}
\bw_\tau (\rho) 
 &= \sharp \{ i \mid a_i \ne 0 \}  - \bv_\tau(\rho) \\
& = \frac{1}{l} \sum _{a_i \ne 0 } l -a_i \\
& = \bv_\tau(\iota(\rho)).
\end{align*}
\end{proof}

From the lemma, 
\begin{align*}
 M(K,\Gamma, - \bw_\tau) 
 & = \frac {1}{\sharp \Gamma}\sum_{\rho\in S_{K,\gamma}}  q ^{\bw_\tau(\rho)} \\
 & = \frac {1}{\sharp \Gamma}\sum_{\rho\in S_{K,\gamma}}  q ^{\bv_\tau(\iota(\rho))} \\
 &= \frac {1}{\sharp \Gamma}\sum_{\rho\in S_{K,\gamma}}  q ^{\bv_\tau(\rho)} \\
 & = \DD (M(K,\Gamma,\bv_\tau)).
\end{align*}
The strong duality thus holds in the tame case. 
Moreover it is a term-wise duality: $q^{\bw_\tau(\rho)}$ and $q ^{\bv_\tau(\iota(\rho))}$ are dual to each
other. This is no longer true in wild cases satisfying dualities.

\subsection{Quadratic extensions: characteristic zero}

We consider an unramified extension $K$ of $\QQ_2$ and compute some total masses counting quadratic extensions of $K$. Namely we consider the case $\Gamma = \ZZ/ 2 \ZZ $. 
Quadratic extensions are divided into 4 classes: let $\fm_K$ be the maximal ideal of $\cO_K$. 
\begin{itemize}
\item The  trivial extension $K\times K$.
\item The  unramified field extension of degree two.
\item Ramified extensions  with discriminant $\fm_K^2$. There are $2(q-1)$ of them.
\item Ramified extensions with discriminant $\fm_K^3$. There are $2q$ of them.
\end{itemize}
For instance, this follows from Krasner's formula \cite{MR0225756} (see also \cite{MR500361}).

They are in one-to-one correspondence with elements of $S_{K,\Gamma}$.
Let $\sigma$ be the two-dimensional representation of $\Gamma = \ZZ/ 2 \ZZ $ by the transposition and for $n\in \NN$, $\sigma_n$ the direct sum of $n$ copies of $\sigma$.
From \cite[Lem.\ 2.6]{Wood-Yasuda-I},  the values  $\ba_{\sigma}(\rho)$, $\rho\in S_{K,\Gamma}$ are respectively $0, 0,  2, 3$
depending on the corresponding class above of quadratic extensions. 
We have
\begin{align*}
M(K,\Gamma,\bv_{\sigma_{n}}) & =M(K,\Gamma,\ba_{\sigma_n}/2) \\
& = \frac{1}{2} (1+1+ 2(q-1)q^{-2n/2}+2q\cdot q^{-3n/2})\\
&= 1+q^{-n+1}-q^{-n}+q^{-3n/2+1}.
\end{align*}

As for the total mass with respect to $-\bw_{\sigma_{n}}$, we have 
$\bw_{\sigma_{n}} (\rho)= n - \bv_{\sigma_{n}}(\rho)$ if $\rho$ corresponds to a totally ramified extension, and $\bw_{\sigma_{n}} (\rho)=\bv_{\sigma_{n}}(\rho)$ otherwise. Therefore,
\begin{align*}
M(K,\Gamma,\bw_{\sigma_{n}}) & = \frac{1}{2} (1+1+ 2(q-1)q^{-n+n}+2q\cdot q^{-3n/2+n})\\
&= q+q^{-n/2+1}.
\end{align*}

We easily see that the strong duality holds for $n=1$. 
For $n>1$, the strong duality does not hold, but the weak one holds.

This phenomenon seems to be rather general. For  a few other examples
of permutation representations $\sigma $ of cyclic groups, the authors checked that the weak duality holds, but the strong duality does not always hold. 
From these computations as well as the example in section \ref{subsec:Bhargava}, 
it is natural to ask:

\begin{question}\label{Q:permutation_dual}
Suppose that the given representation $\tau$ is a permutation representation.
Does the weak duality \eqref{e2} always hold? Does the quotient scheme $\AA_{\cO_K}^d/\Gamma$ always admit an ESWL resolution?
\end{question}

\subsection{Quadratic extensions: characteristic two}

Next we consider the case $K=\FF_q((t))$ with $q$ a power of two
and $\Gamma =\ZZ/2\ZZ$. 
As in the last example, there are exactly two unramified quadratic extensions of $K$: the trivial one $K\times K$ and the unramified field extension.
As for the ramified extensions, for $i \in \NN$, there are exactly 
$2(q-1)q^{i-1}$ extensions with discriminant $\fm_K^{2i}$. There is no extension with discriminant of odd exponent. This again follows from Krasner's formula. 

Let $\sigma$ be the two dimensional $\Gamma$-representation by transposition and 
$\sigma_n$ the direct sum of $n$ copies of it as above. We have
\begin{align*}
M(K,\Gamma,\bv_{\sigma_{n}}) &= 1+ \sum_{i=1}^\infty (q-1)q^{i-1}q^{-in} \\
&= 1 + \frac{(q-1)q^{-n}}{1-q^{-n+1}},
\end{align*}
and
\begin{align*}
M(K,\Gamma,-\bw_{\sigma_{n}}) &= 1+ \sum_{i=1}^\infty (q-1)q^{i-1}q^{-in+n} \\
&= 1 + \frac{q-1}{1-q^{-n+1}},
\end{align*}
We see that the strong duality and hence the weak duality hold for every $n$. 
Thus the answer to Question \ref{Q:positive_char_strong_dual} in this case is positive.

For an integer $m\ge 0$, let us now consider the representation
 $\tau_m:\Gamma \to \GL {2}{\cO_K}$ given by the matrix 
 \[
  \begin{pmatrix}
  1 & t^m \\
  0 & 1
  \end{pmatrix}.
 \]
We have $\tau_0 \cong \sigma$. If $m>0$, then $\tau _m$ is not defined over $k=\FF_q$. 
 
To compute the functions $\bv_{\tau_m}$ and $\bw_{\tau_m}$, let $L$ be a ramified 
quadratic extension of $K$ with discriminant $\fm_K ^{2i}$, $\iota$ its unique $K$-involution  and $\varpi$ its uniformizer. 
Then  $\delta(\varpi):=\iota (\varpi)-\varpi \in \cO_K$ and $v_K( \delta(\varpi))=i$ with $v_K$ the normalized valuation of $K$.  
Therefore the tuning submodule $\Xi$ is computed as follows:
\begin{align*}
\Xi  & = \{ (x,y)\in \cO_L^{\oplus 2} \mid x + t^m y = \iota (x),\, y = \iota (y) \} \\
& = \{ (x,  t ^{-m} \delta (x) ) \mid x\in \cO_L,\, v_K(\delta(x))\ge m  \} \\
& = \{ (x,  t ^{-m} \delta (x) ) \mid x\in \cO_K \cdot 1 \oplus \cO_K \cdot t^a \varpi   \} \\
& = \langle (1,0),\, ( t^a \varpi, t ^{a-m} \delta(\varpi) )  \rangle _{\cO_K},
\end{align*}
where $a = \sup \{  0 , m -i\}$. 
Therefore, 
\[
\frac{\cO_L^{\oplus 2}}{\cO_L  \cdot \Xi} \cong \frac{\cO_L} {\cO_L \cdot t^{a-m} \delta(\varpi)}
\cong \frac{\cO_L}{\cO_L \cdot t^{a-m+i}} .  
\]
For $\rho\in S_{K,\Gamma}$ corresponding to $L$, 
\[
\bv_{\tau_m} (\rho) = a-m+i = 
\begin{cases}
  0 & (i \le m)\\
  i -m & (i > m) .
 \end{cases}
\]

The map $u^* :\cO_K [x,y]\to \cO _L [X,Y]$ in the definition of $\bw$ is explicitly given as
\[
u^*(x) = X + t^a \varpi Y,\, u^* (y) = t^{a-m}\delta(\varpi) Y.
\] 
The induced morphism $\Spec k[X,Y] \to \Spec k[x,y] $ is an isomorphism
if $i \le m$ and a linear  map of rank one if $i>m $. 
Therefore
\[
\bw_{\tau_m}(\rho)= 
\begin{cases}
0 & (i \le m) \\
1 - i + m & (i >m).
\end{cases}
\]
To obtain  quotient morphisms which are étale in codimension one, we consider 
the direct sums $\upsilon_{m,n} := \tau_m \oplus \sigma_n$. 
We have
\begin{align*}
M(K,\Gamma, \bv_{\upsilon_{m,n}}) 
& = 1 + \sum _{i = 1}^{m} (q-1)q^{i-1} \cdot q^{-in} +\sum _{i= m+1}^\infty (q-1)q^{i-1} \cdot q^{- i + m -in}\\
& = 1 + (q-1)q^{-n}\frac{1-q^{(1-n)m}}{1-q^{1-n}}+ (q-1)\frac{q^{m-(m+1)n-1}}{1-q^{-n}}
\end{align*} 
and
\begin{align*}
M(K,\Gamma, - \bw_{\upsilon_{m,n}}) 
& = 1 + \sum _{i = 1}^{m} (q-1)q^{i-1} \cdot q^{-in+n} +\sum _{i= m+1}^\infty (q-1)q^{i-1} \cdot q^{1 - i + m -in+n}\\
& = 1 + (q-1)\frac{1-q^{(1-n)m}}{1-q^{1-n}}+ (q-1)\frac{q^{m-mn}}{1-q^{-n}}
\end{align*} 

One can easily see that the weak duality does not generally hold, for instance, by putting $(m,n)=(1,1)$.
From our viewpoint, this is explained by that the entry $t^m$ in the matrix defining $\tau_m$ 
causes the degeneration of singularities of $\AA_{\cO_K}^d/\Gamma$, namely
makes the family over $\cO_K$ non-equisingular.

\subsection{Bhargava's formula}\label{subsec:Bhargava}

We discuss an example from the preceding paper \cite{Wood-Yasuda-I}
and the duality observed there. Let $S_{n}$ be the $n$-th symmetric
group and $\sigma:S_{n}\to\GL d{\cO_{K}}$ the standard permutation
representation. According to Kedlaya \cite{MR2354797}, Bhargava's
mass formula for étale extensions of a local field \cite{MR2354798}
is formulated as
\[
M(K,S_{n},\ba_{\sigma})=\sum_{m=0}^{n-1}P(n,n-m)\cdot q^{-m}.
\]
Here $P(n,n-m)$ denotes the number of partitions of $n$ into exactly
$n-m$ parts. 

Let 
\[
\tau:=\sigma\oplus\sigma:S_{n}\to\GL{2n}{\cO_{K}}.
\]
From Lemma \ref{lem:permutation}, we have 
\[
M(K,S_{n},\bv_{\tau})=M(K,S_{n},\ba_{\sigma})=\sum_{m=0}^{n-1}P(n,n-m)\cdot q^{-m}.
\]
In \cite{Wood-Yasuda-I}, we verified that
\[
M(K,S_{n},-\bw_{\tau})=\sum_{m=0}^{n}P(n,n-m)\cdot q^{m}.
\]
These show that the functions $M(K_{r},S_{n},-\bw_{\tau})$ and $M(K_{r},S_{n},\bv_{\tau})$
satisfy the strong duality (\ref{eqmd}). Thus the answer to Question \ref{Q:permutation_dual} is positive in this situation.

The quotient scheme $X=\AA_{\cO_{K}}^{2n}/S_{n}$ associated to $\tau$
is nothing but the $n$-th symmetric product of the affine plane over
$\cO_{K}$. It admits a special ESWL resolution, namely the Hilbert-Chow
morphism 
\[
f:\Hilb^{n}(\AA_{\cO_{K}}^{2})\to X
\]
from the Hilbert scheme of $n$ points on $\AA_{\cO_{K}}^{2}$ defined
over $\cO_{K}$. 

Using a stratification of $\Hilb^{n}(\AA_{k}^{2})$ into affine spaces
\cite{MR870732,MR2492446}, we can directly show
\begin{gather*}
\sharp\Hilb^{n}(\AA_{k}^{2})=\sum_{m=0}^{n}P(n,n-m)\cdot q^{2n-m}.
\end{gather*}
These together with Theorem \ref{thm:McKay realized} gives a new
proof of Bhargava's formula without using Serre's mass formula \cite{MR500361}
unlike in \cite{MR2354798,MR2354797}.

\subsection{Kedlaya's formula}

We suppose $p\ne2$ and let $G$ be the group of signed permutation
matrices in $\GL n{\cO_{K}}$. A signed matrix is a matrix such that
every row or every column contains one and only one nonzero entry
which is either 1 or -1. The group is isomorphic to the wreath product
$(\ZZ/2\ZZ)\wr S_{n}$. Let 
\[
\iota:G\hookrightarrow\GL n{\cO_{K}}
\]
be the inclusion and $\tau:=\iota\oplus\iota$.
\begin{lem}
We have $\ba_{\iota}=\bv_{\tau}$. \end{lem}
\begin{proof}
We construct a representation $\iota':G\to\GL{2n}{\cO_{K}}$ by replacing
entries of matrices in $G$ as follows: replace 0 with $\begin{pmatrix}0 & 0\\
0 & 0
\end{pmatrix}$, $1$ with $\begin{pmatrix}1 & 0\\
0 & 1
\end{pmatrix}$ and $-1$ with $\begin{pmatrix}0 & 1\\
1 & 0
\end{pmatrix}$. The obtained $\iota'$ is a permutation representation of $G$,
and from Lemma \ref{lem:permutation}, 
\[
\ba_{\iota'}=\bv_{\iota'\oplus\iota'}.
\]
If $\tau:G\to\GL n{\cO_{K}}$ is the trivial representation, then
$\iota'\cong\iota\oplus\tau$. The facts that $\ba_{\tau}\equiv\bv_{\tau}\equiv0$
and that $\ba_{\iota\oplus\tau}=\ba_{\iota}+\ba_{\tau}$ and $\bv_{\iota\oplus\tau}=\bv_{\iota}+\bv_{\tau}$,
which were proved in \cite{Wood-Yasuda-I}, prove the lemma.
\end{proof}
We then have 
\begin{equation}
M(K,G,\bv_{\tau})=M(K,G,\ba_{\iota})=\sum_{j=0}^{n}\sum_{i=0}^{j}P(j,i)P(n-j)q^{i-n},\label{fK}
\end{equation}
where the right equality is due to Kedlaya \cite[Remark 8.6]{MR2354797} and $P(m)$ denotes
the number of partitions of $m$.

We can construct a resolution of $X:=\AA_{\cO_{K}}^{2n}/G$ as follows.
First consider the involution of $\AA_{\cO_{K}}^{2}=\Spec\cO_{K}[x,y]$
sending $x$ to $-x$ and $y$ to $-y$. Let 
\[
Z:=\AA_{\cO_{K}}^{2}/(\ZZ/2\ZZ)=\Spec\cO_{K}[x^{2},xy,y^{2}]
\]
be the associated quotient scheme. We have a natural isomorphism
\[
X\cong S^{n}Z:=Z^{n}/S_{n}.
\]
Since $Z$ is the trivial family of $A_{1}$-singularity over $\Spec\cO_{K}$
(actually it is a toric variety), there exists the minimal resolution
$W\to Z$ of $Z$ such that $W$ is $\cO_{K}$-smooth, the exceptional
locus is isomorphic to $\PP_{\cO_{K}}^{1}$ and $K_{W/Z}=0$. Let
$\Hilb^{n}(W)$ be the Hilbert scheme of $n$ points of $W$ over
$\cO_{K}$, which is again smooth over $\cO_{K}$. We have the proper
birational morphism 
\[
f:\Hilb^{n}(W)\xrightarrow{\text{Hilbert-Chow}}S^{n}W\to S^{n}Z\xrightarrow{\sim}X
\]
and $K_{\Hilb^{n}(W)/X}=0$. 

We now compute $\sharp f^{-1}(o)(k)$. Let $E\subset W_{k}$ be the
exceptional divisor of $W_{k}\to Z_{k}$, which is isomorphic to $\PP_{k}^{1}$.
Each $k$-point of $f^{-1}(o)$ corresponds to a zero-dimensional subscheme
of $W_{k}$ of length $n$ supported in $E$. There exists a stratification
\[
E=E_{1}\sqcup E_{0}
\]
such that $E_{1}\cong\AA_{k}^{1}$ is a coordinate line of an open
subscheme $\AA_{k}^{2}\cong U\subset W_{k}$ and $E_{0}\cong\Spec k$
is the origin of another open subscheme $\AA_{k}^{2}\cong U'\subset W_{k}$.
Let $C_{m}$ be the set of zero-dimensional subschemes $U$ of length
$m$ supported in $E_{1}$ and $D_{m}$ the set of zero-dimensional
subschemes of $U'$ of length $m$ supported in $E_{0}$. From \cite[Corollary 3.1]{MR2492446},
\[
\sharp C_{m}=P(m)q^{m}
\]
and 
\[
\sharp D_{m}=\sum_{i=0}^{m}P(m,i)q^{m-i}.
\]
Since $f^{-1}(o)(k)$ decomposes as
\[
f^{-1}(o)(k)=\bigsqcup_{j=0}^{n}C_{n-j}\times D_{j},
\]
we have 
\begin{align*}
\sharp f^{-1}(o)(k) & =\sum_{j=0}^{n}\left(P(n-j)q^{n-j}\sum_{i=0}^{j}P(j,i)q^{j-i}\right)\\
 & =\sum_{j=0}^{n}\sum_{i=0}^{j}P(j,i)P(n-j)q^{n-i}.
\end{align*}
From the wild McKay correspondence,
\[
M(K,G,-\bw_{\tau})=\sum_{j=0}^{n}\sum_{i=0}^{j}P(j,i)P(n-j)q^{n-i}.
\]
Comparing this with \eqref{fK},
we verify the strong duality for $M(K_{r},G,-\bw_{\tau})$
and $M(K_{r},G,\bv_{\tau})$.

In a similar way as to compute $\sharp f^{-1}(o)(k)$, we can directly
compute $\sharp\Hilb^{n}(W)(k)$ and obtain a new proof of Kedlaya's
formula (\ref{fK}). 

\section{Concluding remarks and extra problems}\label{sec:Concluding_remarks}

As far as we computed, we observed the following phenomena.
The weak duality and the strong duality may fail.  
Tame representations satisfy both the strong and weak dualities.
Permutation representations satisfy the weak duality but not necessarily the strong duality.
When $K= \FF_q((t))$, representations defined over $\FF_q$ satisfy both the strong and weak dualities. 

 We may interpret these as follows. We might be able to measure by the two dualities 
 how equisingular the family $\AA_{\cO_K}^d/\Gamma$ is. If the strong duality holds, then the family 
 would be very equisingular. If only the weak duality holds, then the family would be moderately so.
 If both dualities fail, then it is rather far from being  equisingular.
It may be interesting to look for a numerical invariant refining this measurement. 

As we saw, the weak duality is derived from the Poincaré duality of stringy invariants if
there exists an ESWL resolution. 
What about the strong duality? Is there any geometric interpretation of it?

Although it is still conjectural, there would be the motivic counterparts of total masses, stringy point counts and the McKay correspondence between them (see \cite{Wood-Yasuda-I,MR3230848,Yasuda:2013fk,Yasuda:2014fk2}). We may discuss dualities in this motivic context as well.
Indeed, some dualities for motivic invariants were verified in \cite{MR3230848,Yasuda:2014fk2}.

\bibliographystyle{alpha}
\bibliography{/Users/Takehiko/Dropbox/Math_Articles/mybib.bib}

\begin{thebibliography}{NSW08}

\bibitem[Bat98]{MR1672108}
Victor~V. Batyrev.
\newblock Stringy {H}odge numbers of varieties with {G}orenstein canonical
  singularities.
\newblock In {\em Integrable systems and algebraic geometry (Kobe/Kyoto,
  1997)}, pages 1--32. World Sci. Publ., River Edge, NJ, 1998.

\bibitem[BD96]{MR1404917}
Victor~V. Batyrev and Dimitrios~I. Dais.
\newblock Strong {M}c{K}ay correspondence, string-theoretic {H}odge numbers and
  mirror symmetry.
\newblock {\em Topology}, 35(4):901--929, 1996.

\bibitem[Bha07]{MR2354798}
Manjul Bhargava.
\newblock Mass formulae for extensions of local fields, and conjectures on the
  density of number field discriminants.
\newblock {\em Int. Math. Res. Not. IMRN}, (17):Art. ID rnm052, 20, 2007.

\bibitem[CV08]{MR2492446}
Aldo Conca and Giuseppe Valla.
\newblock Canonical {H}ilbert-{B}urch matrices for ideals of {$k[x,y]$}.
\newblock {\em Michigan Math. J.}, 57:157--172, 2008.
\newblock Special volume in honor of Melvin Hochster.

\bibitem[ES87]{MR870732}
Geir Ellingsrud and Stein~Arild Str{\o}mme.
\newblock On the homology of the {H}ilbert scheme of points in the plane.
\newblock {\em Invent. Math.}, 87(2):343--352, 1987.

\bibitem[Ked07]{MR2354797}
Kiran~S. Kedlaya.
\newblock Mass formulas for local {G}alois representations.
\newblock {\em Int. Math. Res. Not. IMRN}, (17):Art. ID rnm021, 26, 2007.
\newblock With an appendix by Daniel Gulotta.

\bibitem[Kol13]{MR3057950}
J{{\'a}}nos Koll{{\'a}}r.
\newblock {\em Singularities of the minimal model program}, volume 200 of {\em
  Cambridge Tracts in Mathematics}.
\newblock Cambridge University Press, Cambridge, 2013.
\newblock With a collaboration of S{{\'a}}ndor Kov{{\'a}}cs.

\bibitem[Kra66]{MR0225756}
Marc Krasner.
\newblock Nombre des extensions d'un degr{\'e} donn{\'e} d'un corps
  {${\mathfrak {P}}$}-adique.
\newblock In {\em Les {T}endances {G}{\'e}om. en {A}lg{\`e}bre et {T}h{\'e}orie
  des {N}ombres}, pages 143--169. Editions du Centre National de la Recherche
  Scientifique, Paris, 1966.

\bibitem[LO08a]{MR2434692}
Yves Laszlo and Martin Olsson.
\newblock The six operations for sheaves on {A}rtin stacks. {I}. {F}inite
  coefficients.
\newblock {\em Publ. Math. Inst. Hautes {\'E}tudes Sci.}, (107):109--168, 2008.

\bibitem[LO08b]{MR2434693}
Yves Laszlo and Martin Olsson.
\newblock The six operations for sheaves on {A}rtin stacks. {II}. {A}dic
  coefficients.
\newblock {\em Publ. Math. Inst. Hautes {\'E}tudes Sci.}, (107):169--210, 2008.

\bibitem[Mil80]{MR559531}
James~S. Milne.
\newblock {\em \'{E}tale cohomology}, volume~33 of {\em Princeton Mathematical
  Series}.
\newblock Princeton University Press, Princeton, N.J., 1980.

\bibitem[NSW08]{MR2392026}
J{{\"u}}rgen Neukirch, Alexander Schmidt, and Kay Wingberg.
\newblock {\em Cohomology of number fields}, volume 323 of {\em Grundlehren der
  Mathematischen Wissenschaften [Fundamental Principles of Mathematical
  Sciences]}.
\newblock Springer-Verlag, Berlin, second edition, 2008.

\bibitem[Ols08]{Olsson:2008tv}
M~Olsson.
\newblock {Fujiwara's theorem for equivariant correspondences}.
\newblock {\em preprint}, 2008.

\bibitem[Ser78]{MR500361}
Jean-Pierre Serre.
\newblock Une ``formule de masse'' pour les extensions totalement ramifi{\'e}es
  de degr{\'e} donn{\'e} d'un corps local.
\newblock {\em C. R. Acad. Sci. Paris S{\'e}r. A-B}, 286(22):A1031--A1036,
  1978.

\bibitem[Ser79]{MR554237}
Jean-Pierre Serre.
\newblock {\em Local fields}, volume~67 of {\em Graduate Texts in Mathematics}.
\newblock Springer-Verlag, New York, 1979.
\newblock Translated from the French by Marvin Jay Greenberg.

\bibitem[Sta97]{MR1442260}
Richard~P. Stanley.
\newblock {\em Enumerative combinatorics. {V}ol. 1}, volume~49 of {\em
  Cambridge Studies in Advanced Mathematics}.
\newblock Cambridge University Press, Cambridge, 1997.
\newblock With a foreword by Gian-Carlo Rota, Corrected reprint of the 1986
  original.

\bibitem[Sun12]{MR2950161}
Shenghao Sun.
\newblock {$L$}-series of {A}rtin stacks over finite fields.
\newblock {\em Algebra Number Theory}, 6(1):47--122, 2012.

\bibitem[Woo08]{MR2411405}
Melanie~Matchett Wood.
\newblock Mass formulas for local {G}alois representations to wreath products
  and cross products.
\newblock {\em Algebra Number Theory}, 2(4):391--405, 2008.

\bibitem[WY15]{Wood-Yasuda-I}
Melanie~Matchett Wood and Takehiko Yasuda.
\newblock Mass formulas for local {G}alois representations and quotient
  singularities {I}: a comparison of counting functions.
\newblock {\em International Mathematics Research Notices}, doi:
  10.1093/imrn/rnv074, 2015.

\bibitem[Yasa]{Yasuda:2013fk}
Takehiko Yasuda.
\newblock Toward motivic integration over wild {D}eligne-{M}umford stacks.
\newblock arXiv:1302.2982, to appear in the proceedings of ``Higher Dimensional
  Algebraic Geometry - in honour of Professor Yujiro Kawamata's sixtieth
  birthday".

\bibitem[Yasb]{wild-p-adic}
Takehiko Yasuda.
\newblock The wild {M}c{K}ay correspondence and $p$-adic measures.
\newblock arXiv:1412.5260.

\bibitem[Yasc]{Yasuda:2014fk2}
Takehiko Yasuda.
\newblock Wilder {M}c{K}ay correspondences.
\newblock arXiv:1404.3373, to appear in Nagoya Mathematical Journal.

\bibitem[Yas14]{MR3230848}
Takehiko Yasuda.
\newblock The {$p$}-cyclic {M}c{K}ay correspondence via motivic integration.
\newblock {\em Compos. Math.}, 150(7):1125--1168, 2014.

\end{thebibliography}

\end{document}